\documentclass{amsart}

\newtheorem{theorem}{Theorem}[section]
\newtheorem{lemma}[theorem]{Lemma}
\newtheorem{proposition}[theorem]{Proposition}
\newtheorem{corollary}[theorem]{Corollary}
\newtheorem{definition}{Definition}[section]

\begin{document}

\title[Large deviations and nonhyperbolic Aubry-Mather measures]{Large deviations and Aubry-Mather measures supported in nonhyperbolic closed geodesics}

\author{Artur O. Lopes}
\address{Departamento de Matem\'atica, Universidade Federal do
Rio Grande do Sul, Porto Alegre, RS, Brazil}
\email{arturoscar.lopes@gmail.com}

\author{Rafael O. Ruggiero}
\address{Departamento de Matem\'atica, Pontif\'{\i}cia Universidade Cat\'olica do Rio de Janeiro,
Rio de Janeiro, RJ, Brazil, 22453-900} \email{rorr@mat.puc-rio.br}
\thanks{The second author was partially supported by CNPq, Pronex de Geometria (Brazil), the \'{E}cole Normale Sup\'{e}rieure de Paris
and the Universit\'{e} de Nice. The first author was partially supported by  CNPq, PRONEX -- Sistemas
Din\^amicos, Instituto do Mil\^enio, INCT- IMPA,  Projeto Universal - CNPq and beneficiary of CAPES financial support.}


\date{February 12th, 2009}

\begin{abstract}
We obtain a large deviation function for the stationary measures of
twisted Brownian motions associated to the Lagrangians
$L_{\lambda}(p,v)=\frac{1}{2}g_{p}(v,v)-
\lambda\omega_{p}(v)$,
where $g$ is a $C^{\infty}$ Riemannian metric in a compact surface
$(M,g)$ with nonpositive curvature, $\omega$ is a closed 1-form
such that the Aubry-Mather measure of the Lagrangian
$L(p,v)=\frac{1}{2}g_{p}(v,v)-\omega_{p}(v)$ has support in a
unique closed geodesic $\gamma$; and the curvature is negative at
every point of $M$ but at the points of $\gamma$ where it is zero.
We also assume that the Aubry set is equal to the Mather set. The
large deviation function is of polynomial type, the power of the
polynomial function depends on the way the curvature goes to zero
in a neighborhood of $\gamma$. This results has interesting
counterparts in one-dimensional dynamics with indifferent fixed
points and convex billiards with flat points in the boundary of
the billiard.  A  previous estimate by N. Anantharaman of the
large deviation function in terms of the Peierl's barrier of the
Aubry-Mather measure is crucial for our result.
\vspace{0.2cm}

AMS 37D40, 37C50, 37A50, 37D05, 53D25

\vspace{0.2cm}
To appear in DCDS A
\end{abstract}
\maketitle

\section{Introduction}

Large deviations of families of measures are important in many physical applications
where one would like to estimate observables at exceptional conditions using ''physical'' measures
in the phase space. A typical setting in dynamical systems would be to estimate
an invariant measure supported in a singular set with respect to the Lebesgue measure
(for instance, a measure minimizing some variational principle) in terms of a family of
absolutely continuous measures containing a sequence converging
to this singular measure. There is a vast literature on the subject in mathematical
physics, assuming in most of the cases a hyperbolicity condition for the dynamical system and/or dimension one for the
configuration space. The subject of the present article is to study large deviations in the
non-hyperbolic setting and higher dimensions with the help of weak KAM theory (see \cite{A1} \cite{A2} \cite{A3} \cite{AIPS} for related results).
We will consider a family of surfaces initially described in \cite{CR}
which have negative curvature everywhere up to the points along a closed geodesic where the curvature vanishes. Since the hyperbolicity of orbits arises from
non-parallel Jacobi fields, and the curvature vanishes along $\gamma$, the orbit
corresponding to $\gamma$ is not hyperbolic.

As a motivation for studying such problem we point out that there classes of problems (which do not present full hyperbolicity)
where a special orbit plays an important role. For transformations with a fixed indifferent point (like the Maneville-Pomeau map, see \cite{You}, \cite{Sa}) this point is associated with the phenomena of phase transition and polynomial decay of correlation. For special billiards, where a cusp point can make a trajectory stay for arbitrary long time close to this point, this is also associated with polynomial decay of correlation \cite{Mar}. The careful analysis of the evolution on time of these special trajectories determines singular characteristics of these special cases of dynamics. Here we analyze the phenomena
of large deviation associated to a special closed geodesic. The family of probabilities indexed by $\lambda$ which is considered here is associated to critical solutions of the Evans action (see \cite{Ev1} \cite{GLM2}), and, the limit, when $\lambda\to \infty$, is usually known as the semiclassical limit (see \cite{A3} \cite{Non}).

 One of the important results we get is the link between Peierls barrier and Busemann functions, which allow us to make accurate analytic estimates of the former from the geometry of the surface closed to the vanishing curvature geodesic $\gamma$. We point out that our work considers a pathological case, and some kind of analytical control over the lack of hyperbolicity is essential to get meaningful large deviation estimates. This is similar to the case of the investigation of the ergodic properties of dynamical systems with fixed indifferent points \cite{You} \cite{Sa} or billiards with cusps \cite{Mar}.

 We denote by $d$ the distance on the  manifold induced by the Riemannian metric.

From these estimations we get  the main result of the paper which is the following:
\bigskip

{\bf Theorem 1.}
 Let $(M,g)$ be a compact surface with $K\leq 0$ such that:
\begin{enumerate}
 \item There is a closed geodesic $\gamma$ where $K\equiv  0$
whose orbit supports the (unique) Aubry-Mather measure of
$L(p,v)=\frac{1}{2}g_{p}(v,v)-\omega_{p}(v)$. The Aubry set and
the Mather set of the Aubry-Mather measure coincide.
\item $K<0$ in the complement of $\gamma$.
\item There exists $m>0$ such that for every geodesic
$\beta: (-\epsilon, -\epsilon) \longrightarrow M$
 perpendicular to $\gamma$ at $\beta(0)=\beta \cap \gamma$
we have that $m$ is the least integer where
 $\frac{\partial^{m}}{dt^{m}} K(\beta(t))\vert_{t=0}\neq 0.$
\end{enumerate}
$\mu_{\lambda \omega}$ be the stationary measure for the
$\lambda\omega$-twisted Brownian motion, $\lambda >0$. Then
$\exists$ $D>0$ such that $\forall$ open   ball $A\subset M$
(which does not intersect the closed geodesic $\gamma$),
$$ -(1/D)\inf_{x\in A}d(x,\gamma)^{2+\frac{m}{2}}  \leq  \lim_{\lambda \rightarrow +\infty}\frac{1}{\lambda}\ln(\mu_{\lambda\omega}(A))$$
and
$$ \lim_{\lambda \rightarrow +\infty}\frac{1}{\lambda}\ln(\mu_{\lambda\omega}(\bar{A}))
  \leq -D\inf_{x\in \bar{A}}d(x,\gamma)^{2+\frac{m}{2}}.$$
\bigskip

To a given form $\omega$, corresponds by duality in Mather Theory, a homology class $[h]$
(see for instance \cite{CI} \cite{Man} \cite{Mat}). In our case this $[h]$ is the
homology of the curve $\gamma$. A brief account of some basic results of
Brownian motion and Aubry-Mather theory is made in the first three sections of the paper.
\bigskip

Several precise results about large deviations for hyperbolic dynamical systems
are known \cite{A2} \cite{AIPS}. In the case of a probability supported on a
proper invariant hyperbolic set, the large deviation function is typically a
linear function in the distance to the support of the measure;
its slope depending on the hyperbolicity of the system (this in the
context of Lagrangian dynamics on the torus see also \cite{A1}). The thermodynamical
formalism is the main tool used to get large deviation formulae in the
presence of hyperbolicity.
\bigskip

Let us give an outline of the proof of Theorem 1. N. Anantharaman in
\cite{A1} \cite{A2} and \cite{A3} considers the family of measures
mentioned in Theorem 1: stationary measures for the twisted Brownian
motions arising from twisting a Riemannian Laplacian by multiples of
a closed one form $\omega$. By the general theory of harmonic
analysis, such measures are absolutely continuous probabilities on
the configuration space (the manifold $M$). These probabilities
approximate the Aubry-Mather measure associated to the Lagrangian
given by the kinetic energy of a Riemannian metric plus the closed
form $\omega$. Under certain assumptions (uniqueness of the
Aubry-Mather measure), N. Anantharaman also considers a large
deviation principle for this family,  and exhibits a deviation
function which is given by the Peierl's barrier. This nice geometric
result tells roughly that the deviation function at a point $p$
depends on how far from being minimizers of the Lagrangian action
are closed loops based at $p$. We consider here an special example
of this setting where one can have a sharp control of the Peierl압
barrier in a neighborhood of a certain closed, non-hyperbolic
geodesic (the curve $\gamma$ in Theorem 1). In \cite{AIPS}, among
other things,  it is analyzed a similar problem for a hyperbolic
periodic trajectory  of the geodesic flow. General references for
the Aubry-Mather theory are \cite{Mat} \cite{Man} \cite{CI}
\cite{Fa}.
\bigskip

Our estimate of the Peierl압 barrier comes from sharp bounds for the
Busemann function (see \cite{Bus}) of the Riemmanian metric $(M,g)$ associated to lifts
of the geodesic $\gamma$ in the universal covering. These bounds are obtained by comparing
the metric $(M,g)$ in a neighborhood of $\gamma$ with an annulus of revolution (under the
assumptions of Theorem 1). Notice that the Peierl압 barrier is defined in terms of the
weak KAM solutions of the considered Lagrangian (see Sections 4 and subsequent). Therefore, one of the
main issues of the proof of Theorem 1 is to relate Busemann functions and Peierl압 barrier.
There is a natural generalization of Busemann functions to
convex, superlinear Lagrangians (see section 4.9 of \cite{CI} for instance) in the context
of weak KAM theory. Such generalized Busemann functions are used to exhibit fixed
points of the Lax-Oleinik operator (backward and forward)
with infinite critical value (in \cite{CI} \cite{Co1} these
functions are called Busemann weak KAM solutions). We would like to point out that
the Busemann functions we use are just the Riemannian ones, we do not need to apply
this general notion in our argument.
\bigskip

Let us make some comments about the assumption in Theorem 1 concerning the equivalence between
the Aubry-Mather set and the Mather set. This condition is of topological nature, as
observed in \cite{Mas} for surfaces. Particularly important for us are the results
of section 4 in \cite{Mas} where it is considered the case where the Mather set is a
single periodic orbit: if $\gamma$ separates $M$ (see case 1.2) then the two sets are equal.
In fact, the coincidence of the Aubry-Mather set and the Mather set is generic in homology
(see Theorem 3 in \cite{Mas}).
\bigskip

Let us finish the Introduction with some further remarks and problems.
First of all, notice that Theorem 1 describes how a family of absolutely continuous measures in the
{\bf configuration} space approaches a Dirac measure concentrated in the so-called
projected Mather set (see Section 3 for the definition). The terminology Mather measure may have
two meanings: the projected Mather measure (with support on the manifold)  and the one in the tangent bundle.
So Theorem 1 is  a large deviation
for the projected Mather measure. A more recent stream of ideas
give us some hints of how we can obtain a L. D. P. for the Mather measure on the tangent bundle.
The so-called entropy penalized method presented in \cite{GV} shows other
ways to obtain approximations of the Mather measure on the tangent bundle by
absolutely continuous probabilities. Under some conditions ($M=T^{n}$, convex superlinear
Lagrangians, uniqueness of the Aubry-Mather measure), a large deviation principle
for such procedure is described in \cite{GLM}. Another way to get a L. D. P.
for the Mather measure on the tangent bundle is via semi-classical limit of
Wigner functions \cite{GLM2}.  We believe that these procedures can also be
extended to our setting.
\bigskip

Finally, let us observe that the geodesic flows considered here are particular cases of expansive geodesic flows
in manifolds with non-positive curvature \cite{E} \cite{BGS} \cite{Ru1} \cite{Ru2} \cite{CR} \cite{LRR}.
The topological dynamics of expansive geodesic flows in manifolds without conjugate points is
well understood, it is about the same Anosov topological dynamics. Moreover, the universal
covering is a Gromov hyperbolic space. However, the ergodic theory of expansive geodesic
flows with non-positive curvature is almost the same ergodic theory of rank one manifolds.
The ergodicity of rank one manifolds is a very hard open problem, as well as for expansive
geodesic flows with non-positive curvature. The family of examples considered in Theorem 1
are perhaps the simplest non-Anosov geodesic flows of rank one, they are ergodic an even Bernoulli
by Pesin theory. So we could ask if it is possible to give a sharper description of the ergodic
properties of invariant measures in this case (Liouville measure, Gibbs measures).
In \cite{GN} the Holder class of the horocycle flow for the metrics considered in Theorem 1 is
presented. We believe that from these estimates, and some of the ideas described in the present paper,
one can detect the so called  concentration of measure phenomena, or even calculate the decay of correlation
of the Liouville measure. This will be the purpose of a future work.
\bigskip

The second author thanks the \'{E}cole Normale Sup\'{e}rieure de Paris where most of
this work was done during a sabbatic period of the author. Special thanks to Professor
Viviane Baladi who made possible the visit of the second author to the ENS Paris. Special
thanks too to the Mathematics Department of the Universit\'{e} de Nice,
where the second author developed part of this work when invited by Professor Ludovic Rifford.







We refer the reader to \cite{DZ} for general properties of large deviations.

\section{Preliminaries about diffusion and Brownian motion}
\subsection{Diffusion and Brownian motion in a Riemannian manifold}

Let $(M,g)$ be a compact $C^{\infty}$ Riemannian manifold, let
$(\tilde{M},\tilde{g})$ be its universal covering endowed with the
pullback of $g$ by the covering map. Let $\Delta$ be the Laplace
operator of $(M,g)$, $\tilde{\Delta}$ be the lift of the Laplace
operator to $(\tilde{M},\tilde{g})$.

The operator $\Delta$ gives rise to a stochastic process in
$(M,g)$, the Brownian motion. It is linked to $\Delta$ in the
following way: given $x\in M$, let $C({\mathbb R}, M)$ be the
space of continuous paths, ${\mathbb P}_{x}$ be the Wiener
measure, let $X_{t}: \gamma \longrightarrow \gamma(t)$ be a
realization of the Brownian motion starting at $x$. Then
\[ P^{t}(f(x))= e^{\frac{1}{2}t\Delta}(f(x))= {\mathbb E}_{x}(f(X_{t})) .\]
for every $C^{\infty}$ function $f: M\longrightarrow {\mathbb R}$.
The operator $P^{t}$ is the {\bf Heat semigroup} of $(M,g)$, i.e.,
the solution of the heat equation $\frac{\partial u}{\partial t} =
\frac{1}{2}\Delta u$ in $(M,g)$.

We refer the reader to \cite{Du} or \cite{Str} for general results
on Brownian motion and diffusions.

\subsection{Twisting the Laplacian by closed 1-forms}

We refer the reader to \cite{A1} for general results on twisted Laplacians.

Let $\omega$ be a $C^{\infty}$ closed 1-form in $M$. The Laplacian
\textbf{twisted} by the 1-form $\omega$ is
\begin{eqnarray*}
 \Delta_{\omega}(f(x)) & = & \tilde{\Delta}_{\tilde{x}}(\tilde{f}(\tilde{x})) \\
& = &
e^{-\int_{p}^{\tilde{x}}\tilde{\omega}}\tilde{\Delta}(e^{\int_{p}^{\tilde{x}}\tilde{\omega}})\tilde{f}(x).
\end{eqnarray*}
where $p\in \tilde{M}$ is a base point, $\tilde{\omega}$ is any
lift of $\omega$ to $\tilde{M}$, $f$ is a $C^{\infty}$ function in
$M$ and $\tilde{f}$ is any lift of $f$. Twisted Laplace operators
are used to study asymptotic properties of the number of closed
orbits in a fixed homology class of geodesic flows of negative
curvature (via Selberg's trace formula) taking $Re(\omega)=0$. The
semigroup
\[ P_{\omega}^{t} = e^{\frac{1}{2}\Delta_{\omega}t} \]
gives the solution of the \textbf{twisted} heat equation $
\frac{\partial u}{\partial t} = \frac{1}{2}\Delta_{\omega} u$ for
the {\bf twisted Lagrangian} $ L_{\omega}(p,v) =
\frac{1}{2}g_{p}(v,v) - \omega(v) $.

The twisted Laplacian  appears in a natural way when we want to consider the Schrodinger operator
for a Mechanical Lagrangian to which we add a closed form (the magnetic term).

\subsection{Stationary probability for the twisted Brownian motion}

The operator $P^{t}_{\omega}$ acts on the space of measures by
$(P^{t}_{\omega})^{*}$:
\[ \int_{M} fd(P^{t}_{\omega})^{*}\mu) = \int_{M} P_{\omega}^ {t}f d \mu .\]
The action preserves positive measures and there exist
$\Lambda(\omega)$ and a measure $\mu_{\omega}$ such that
\[ (P^{t}_{\omega})^{*}\mu_{\omega} = e^{\Lambda(\omega)t}\mu_{\omega} .\]
Let $h_{\omega}(x) = \int_{M}K^{1}_{-\omega}(y,x)d\mu_{-\omega}$,
where $K^{t}_{\omega}$ is the Kernel of $P^{t}_{\omega}$.
\bigskip

\begin{theorem} There exists a (unique up to normalization) measure
$\nu_{\omega}= f_{\omega}dx$ that is a fixed point of the twisted
Brownian motion
\[ Q^{t}f(x) = e^{-t\Lambda(\omega)}h_{\omega}(x)^{-1}P^{t}_{\omega}(h_{-\omega}f)(x) .\]
\end{theorem}
The measure $\nu_{\omega}$ will be called the stationary measure
of the twisted Brownian motion.

\section{Preliminaries about Aubry-Mather measures}
\subsection{Aubry-Mather measures}
Consider $M$ a compact $C^{\infty}$ Riemannian manifold.
Let $L:TM\to {\mathbb R}$ be a $C^{\infty}$ convex,
superlinear Lagrangian. The action of $L$ in an absolutely
continuous curve $c:I\to M$ is $A_{L}(c)=
\int_{I}L(c(t),c'(t))dt$.

Let $\mathcal{M}(L)$ be the set of
invariant probability measures of the E-L flow of $L$.
The {\bf action} of $L$ in $\mathcal{M}(L)$ is defined by
\[ A_{L}(\mu) = \int Ld\mu .\]
The {\bf homology class} (Mather, Ma\~{n}\'{e}) $\rho(\mu)$ of the
measure $\mu$ is given by
\[  <\rho(\mu), \omega> = \int \omega d\mu ,\]
where $\omega$ is a closed 1-form. (Recall that the homology group
$H_{1}(M,{\mathbb R})$ is the dual of the cohomology group
$H^{1}(M,{\mathbb R})$).

A measure $\mu \in \mathcal{M}(L)$ is called \textbf{minimizing in
its homology class} if
\[ A_{L}(\mu)= \inf\{ A_{L}(\nu), \rho(\nu)=\rho(\mu)\}. \]

An \textbf{Aubry-Mather measure} $\mu$ is defined by
\[ A_{L}(\mu)= \inf\{A_{L}(\nu), \nu \in \mathcal{M}(L)\} .\]

The union of the supports of all Aubry-Mather measures is called
the Mather set for $L$.

\begin{theorem}\label{graphminimizing} (Mather, Ma\~{n}\'{e}): The support of a
minimizing measure is a Lipschitz graph over an invariant set of
global minimizers of the action.
\end{theorem}

\subsection{Critical energy values and minimizing measures}

Both globally minimizing measures in homology and Aubry-Mather measures arise as minimum (and hence critical)
points of the Lagrangian action on holonomic invariant measures. Moreover, the support of an ergodic invariant
measure has constant energy, and the support of globally minimizing measures are minimizing orbits of the
Euler-Lagrange flow by Theorem \ref{graphminimizing}. So it is natural to expect that the energy levels of the
supports of Aubry-Mather measures have critical properties somehow.

\begin{definition}
The critical value $c(L)$ of $L$ (see section 2.1 \cite{CI}) is defined by
\[ c(L) = \sup_{k\in {\mathbb R}}\{ A_{L+k}(\beta)<0 \mbox{ }\textit{for some} \mbox{ }\textit{closed}\mbox{ }\textit{curve}\mbox{ }\beta\} .\]
\end{definition}

An holonomic invariant measure $\mu$ is called \textbf{globally minimizing} if $A_{L}(\mu) = -c(L)$.

\begin{definition}
The strict critical value $c_{0}(L)$ of $L$ (see page 798 \cite{CIPP}) is given by
\[ c_{0}(L)= \min_{\alpha \in H^{1}(M,{\mathbb R})}c(L-\alpha) = -\beta(0),\]
where $\beta: H_{1}(M,{\mathbb R}) \longrightarrow {\mathbb R}$ is
\[ \beta(h) = \min_{\rho(\mu)=h}A_{L}(\mu). \]
\end{definition}

Notice that $c_{0}(L) \geq c(L)$. The strict critical level is the
relevant one regarding Aubry-Mather measures.

\begin{theorem} \cite{Car}: The support of an Aubry-Mather
measure is contained in the energy level $E=c_0(L)$.
\end{theorem}

There are many equivalent geometric characterizations of the strict critical level, it is, for instance,
the infimum of the energy levels containing a globally minimizing
orbit of the Lagrangian action with nontrivial (real) homology class. This is why the strict critical value is
the critical value of the lift of the Lagrangian action to the abelian cover of the manifold (see \cite{Fa}, for instance).

Let us give some examples. The critical value of geodesic flows is clearly $0$ while the strict critical value
is nonzero, if and only if, the first homology group of the manifold is nontrivial. The critical value
of a mechanical Lagrangian is the opposite value of the maximum of the potential, and the
Euler-Lagrange flow in energy levels above this value can be reparametrized to give the geodesic flow of a Riemannian
metric (Maupertuis' principle). The strict critical value of $L(p,v)=\frac{1}{2}g_{p}(v,v)-\omega_{p}(v)$, where $\omega$
is a closed 1-form, is $\frac{1}{2}\parallel \omega \parallel_{s}^ {2}$, where $\parallel
\omega \parallel_{s}$ is the stable norm. In particular, under our hypothesis, the Aubry-Mather measure of
$L(p,v)=\frac{1}{2}g_{p}(v,v)-\omega_{p}(v)$
is supported in the closed orbit $(\gamma(t),\gamma' (t))$, $t \in [0,Per(\gamma)]$, and the form $\omega$
is dual to the homology class of $\gamma$. So the stable norm of $\omega$ is just the period of $\gamma$.

\section{Large Deviations and Weak KAM theory}

\subsection{Large Deviations of stationary measures and Peierl's barrier}

In this section we state the main tool we use to obtain the
deviation function for the stationary measures of Brownian motions
twisted by multiples of a close 1-form. Let us start with some
basic analytic definitions in the context of weak KAM theory. Our main
references are \cite{Fa} \cite{CI}. Through the section, $M$ will
be a compact $C^{\infty}$ manifold, and $L:TM\longrightarrow
{\mathbb R}$ will be a $C^{\infty}$ convex, superlinear
Lagrangian.

\begin{definition} {\bf The Lax-Oleinik operators $T^{-}_{t}$, $T^{+}_{t}$}.
Given a continuous function $f:M\to(-\infty, \infty)$,
and $t>0$, define the function $T^{-}_{t}(f): M\to
{\mathbb R}$ by
$$ T^{-}_{t}(f)(x)= \inf_{\gamma}[ f(\gamma(0)) +
\int_{0}^{t}L(\gamma(t),\gamma'(t)) dt ],$$ where
$\gamma:[0,t]\to M$ is an absolutely continuous curve
with $\gamma(t)=x$. Let
$$T^{+}_{t}(f)(x) = \sup_{\gamma}[f(\gamma(t)) -
\int_{0}^{t}L(\gamma(t),\gamma'(t)) dt ],$$ where
$\gamma:[0,t]\to M$ is an absolutely continuous curve
with $\gamma(0)=x$.
\end{definition}

The Lax-Oleinik operators $T^{-}_{t}$ form a continuous time semigroup family of
operators, as well as the operators $T^{+}_{t}$. They enjoy very
nice properties (see \cite{Fa1}), in particular, the family of
$T^{-}_{t} -c_{0}(L)t$ has a fixed point $u^{-}$ which is a viscosity
solution of the Hamilton-Jacobi equation (see Definition 7.2.3 and Proposition 7.2.7 in \cite{Fa} and also \cite{CI}).

It is also true that  $u^{-}$ is a Lipschitz function (see Theorem 4.4.6 and Corollary 4.4.13 \cite{Fa}), and Lebesgue almost everywhere we have that
$$H(x,d_{x}u^{-}) = c_{0}(L) .$$

The function $u^{-}$ is  is differentiable along the projection
$\mathcal{M}_{0}$ of the Mather set.  Analogously, the family of
operators $T^{+}_{t} +c_{0}(L)t$ has a fixed point $u^{+}$ which is
Lipschitz and a weak (in the above sense) solution of  a
Hamilton-Jacobi equation (the so called conjugated Hamilton-Jacobi
equation) \cite{Fa}. Let $S^{-}$ be the set of fixed points of
$T^{-}_{t} -c_{0}(L)t$, and let $S^{+}$ be the set of fixed points
of $T^{+}_{t} -c_{0}(L)t$.

\begin{definition}
A pair of functions $u^{-} \in S^{-}$, $u^{+} \in S^{+}$ is called
a conjugate pair if $u^{-}(x)= u^{+}(x)$ for every $x \in \mathcal{M}_{0}$.
\end{definition}

According to Fathi \cite{Fa1} \cite{Fa}, for each weak solution
$u^{-}$ of the Hamilton-Jacobi equation there exists a unique
solution $u^{+}$ such that $u^{-}$, $u{+}$ form a conjugate pair.
We have that the differences $u^{-}-u^{+}$ of conjugate pairs are
always nonnegative, and they vanish at the projected Mather set.
One of the most interesting questions in weak KAM theory is wether
the set of zeroes of the differences of conjugate pairs coincides
with the projected Mather set. In this case, we can characterize
the projected Mather set (and hence the set of global minimizers
of the action) as the set of {\it true} critical points of the difference of
two $C^{1}$-{\it smooth} sub-solutions of the Hamilton-Jacobi equation
\cite{FaSi}. This property is very useful for applications, and motivates the key
idea of the proof of our main theorem.

So it looks very tempting to try to characterize analytically the projected Ma\~{n}\'{e} set
in terms of the differences of conjugate pairs. However, we have to be careful in this point.

\begin{definition} The second Peierl's barrier is

\[ P(x,x)= \inf\{u^{+}(x)-u^{-}(x) \} \]
where the infimum is taken over all conjugate pairs $u^{+}$,
$u^{-}$.
\end{definition}

\begin{definition}
As in \cite{Fa}, we denote the projected Aubry set by
$\pi(\hat{\Sigma} (L))$, the intersection of the set of
zeroes of all conjugate pairs of the Hamilton-Jacobi equation.
\end{definition}

The projected Aubry set is the canonical projection of certain orbits of the Euler-Lagrange flow
(see \cite{Fa} for instance), whose union is called the Aubry set $\hat{\Sigma}(L)$. Moreover,
the canonical projection $\pi:  \hat{\Sigma}(L) \longrightarrow \pi(\hat{\Sigma}(L))$ is a Lipschitz
homeomorphism (see \cite{Fa}, this is a version of the well known Mather's graph Theorem).

\begin{lemma}
The projected Aubry set contains $M_{0}$.
\end{lemma}

However, the inclusion of $M_{0}$ in $\pi(\hat{\Sigma} (L))$ might be strict: this is often the case when
there exist globally minimizing connections between different non-wandering components of the Mather set.
The projected Aubry set has the analytic characterization we would like to have for the Mather set
(see \cite{Fa}, \cite{Fa1}, \cite{A2} \cite{A3} and section 3.7 in \cite{CI}).

\begin{proposition}
The second Peierl's  barrier $P(x,x)$  is zero, if and only if, $x$ is in the projected Aubry set.
\end{proposition}

Finally, we are able now to state the main result of the section, which is one of the main tools
used in the proof of Theorem 1.

\begin{theorem} \label{peierldeviation} \cite{A2} \cite{A3} Let $(M,g)$ be a compact
surface with $K\leq 0$, let $\gamma$, $\omega$ be as in the
assumptions of the main theorem. Then, the measures
$\mu_{\lambda\omega}$ satisfy the following large deviation type
formulae
\[ \lim_{\lambda \rightarrow +\infty}\frac{1}{\lambda}ln(\mu_{\lambda\omega}(A)) \leq  -\inf_{x\in A}P(x,x),\]
for every closed set $A \in M$, and
\[ \lim_{\lambda \rightarrow +\infty}\frac{1}{\lambda}ln(\mu_{\lambda\omega}(B)) \geq  -\inf_{x\in B}P(x,x),\]
for every open set $B$.
\end{theorem}

In other words, we know  that the measures $\mu_{\lambda\omega}(A)$ tend to zero as $\lambda \rightarrow \infty$; so,
the Peierl's barrier gives an estimate of the logarithmic rate of convergence. The logarithmic rates became worst
as the sets approach the Aubry set, and, if the closure of an open set meets the Aubry set, then the above rates
are just $0$. This result can be interpreted in the present  situation as a concentration of the stationary measures around a Dirac
measure in the space of continuous paths which assigns measure one to the closed geodesic in the Mather set, and zero to any
set not containing this geodesic
We stated Theorem \ref{peierldeviation}
suited to our purposes, as it is in \cite{A3}: we are assuming that the Aubry set and the Mather set coincide,
and that there is a unique Mather measure. A more general result in \cite{A2} grants that only the first one
of the inequalities in Theorem \ref{peierldeviation} holds.

\section{Busemann functions}

Manifolds with nonpositive curvature are special examples of
manifolds without conjugate points (the exponential map at every
point is nonsingular). So every geodesic in $\tilde{M}$ is
globally minimizing, and the convexity of the metric yields the
existence of two lagrangian, invariant foliations whose leaves are
locally graphs of the canonical projection. A well known way to
define such foliations in terms of closed 1-forms is through the
so-called {\em Busemann functions}:
given \( \theta = (p,v)\mbox{ } \in \mbox{ } T_{1}\tilde{M} \) the
{\em Busemann function } \( b^{\theta}\mbox{ }: \tilde{M} \to
\mathbb{R}\) associated to \( \theta\) is defined by \[
b^{\theta}(x) = \lim_{t\rightarrow
+\infty}(\tilde{d}(x,\gamma_{\theta}(t)) - t) ,\]
where $\tilde{d}$ is the metric on $\tilde{M}$. From now on we will also denote such distance by $d$.

The level sets of
\(b^{\theta}\) are the {\em horospheres} \( H_{\theta}(t) \) where
the parameter $t$ means that \( \gamma_{\theta}(t) \mbox{ } \in
\mbox{ } H_{\theta}(t) \) (notice that \( \gamma_{\theta}(t) \)
intersects each level set of \(b^{\theta}\) perpendicularly at
only one point). The next lemma summarizes some basic properties of
horospheres (which can be found in \cite{Ru1} \cite {Ru2} \cite{E},
for instance).

\begin{lemma} Let $(M,g)$ be a compact $C^{\infty}$ manifold
without conjugate points.
\begin{enumerate}
\item \( b^{\theta}\) is a \( C^{1} \) function for every
$\theta$. If $(M,g)$ has nonpositive curvature $b^{\theta}$ is a
$C^{2}$ function for every $\theta$.
\item The gradient \( \nabla b^{\theta} \) has norm equal to one at every point.
\item Every horosphere is a $C^{1+K}$, embedded submanifold of dimension $n-1$
($C^{1+K}$ means $K$-Lipschitz normal vector field), where $K$ is
a constant depending on curvature bounds. If the curvature of
$(M,g)$ is nonpositive each horosphere is a $C^{2}$ submanifold.
\item The orbits of the integral flow of \( -\nabla b^{\theta} \),
\( \psi^{\theta}_{t} : \tilde{M} \longrightarrow \tilde{M} \), are
geodesics which are everywhere perpendicular to the horospheres
$H_{\theta}$. In particular, the geodesic \( \gamma_{\theta} \) is
an orbit of this flow and we have that
\[ \psi^{\theta}_{t}( H_{\theta}(s)) \mbox{ }= \mbox{ } H_{\theta}(s+t) \]
for every $t,s \in R$.
\end{enumerate}
\end{lemma}

\bigskip

A geodesic $\beta$ is {\em asymptotic} to a geodesic $\gamma$ in
$\tilde{M}$ if there exists a constant $C>0$ such that
$d(\beta(t), \gamma(t)) \leq C$ for every $t \geq 0$. Nonpositive
curvature implies  that every two integral orbits of $ -\nabla
b^{\theta} $ are asymptotic. Item (2) tells us that Busemann
functions are special, exact solutions of  the Hamilton-Jacobi
equation
\[\tilde{ H}(p,d_{p}b^{\theta})= 1,\]
where $\tilde{ H}:T^{*}\tilde{M} \to {\mathbb R}$ is just the
pullback $\tilde{g}$ of the metric $g$ by the covering map,
$ \tilde{ H}(p,v)= \frac{1}{2}\tilde{g}_{p}(v,v)$. Moreover, if $(M,g)$ is a
compact surface without conjugate points and $\gamma_{\theta}
\subset \tilde{M}$ is a lift of a closed geodesic, the set of
points $x\in \tilde{M}$ where $b^{\theta}(x)+b^{-\theta}(x) =0$ is
just the set of lifts of closed geodesics homotopic to
$\pi(\gamma_{\theta})$ which are axes of $T_{\gamma_{\theta}}$.
This elementary observation is crucial for the section: when
$\pi(\gamma_{\theta})$ is unique in its homotopy class the
functions $b^{\theta}$, $b^{-\theta}$ behave like a pair of
conjugate solutions of the Lax-Oleinik operator. Namely,
$b^{\theta}+b^{-\theta}$ takes its minimum value zero just at the
points of $\gamma_{\theta}$, and it is positive everywhere else.

\section{Peierl's barrier and Busemann functions}

\begin{definition}
The Peierl's barrier is the function $h:M\times M \longrightarrow
{\mathbb R}$ given by
$$h(x,y) = \lim_{T\rightarrow \infty} inf_{\alpha \in C_{T}(x,y)}\{ A_{L+c_{0}(L)}(\alpha)\} ,$$
where $C_{T}(x,y)$ is the set of $C^1$ curves $\alpha: [0,T]\longrightarrow M$ such
that $\alpha(0)=x$, $\alpha(T)=y$.
\end{definition}

The above definition given by R. Man\'{e} \cite{Man} is based in
an analogous definition due to J. Mather \cite{Mat}. We introduce
the Peierl's barrier by two reasons. First of all, it is naturally
connected to the second Peierl's barrier defined in the previous
section; and secondly, the Peierl's barrier is defined in a more geometric way than
the second Peierl's barrier; it's actually really close to Busemann functions.
The purpose of the section is to describe in detail the major issues. Let us begin
with some basic properties of the Peierl's barrier (see \cite{Fa}).

\begin{lemma} \label{PeierlAubry}
Let $M$ be a $C^{\infty}$ compact manifold, and $L$ be a $C^{2}$ Lagrangian that is strictly convex
and superlinear in each tangent space $T_{p}M$. Let $h(x,y)$ be the Peierl's barrier of $L$. Then
\begin{enumerate}
\item The function $h(x,y)$ is Lipschitz continuous.
\item A point $x$ is in the projected Aubry set if and only if $h(x,x)=0$.
\item Given $x \in M$, there exists a sequence $\gamma_{n}:[0,t_{n}]\longrightarrow M$
of minimizers of the action of $L$ such that
\begin{itemize}
\item $\gamma_{n}(0)=\gamma_{n}(t_{n})=x$ for every $n >0$.
\item $\lim_{n\rightarrow \infty}t_{n} = \infty$.
\item $h(x,x) = \lim_{n\rightarrow \infty}(A_{L+c_{0}(L)}(\gamma_{n}))$.
\end{itemize}
\end{enumerate}
\end{lemma}

So $h(x,x)$ vanishes at the projected Aubry set, like the differences of conjugate pairs of the
Hamilton-Jacobi equation. The geometry of the manifold shows up in item (2) of the above lemma:
if $L(p,v)= \frac{1}{2}g_{p}(v,v) - \omega$, the minimizers are just geodesics of $(M,g)$, so item
(2) tells us that the value of $h(x,x)$ is the limit of the values of the action of $L+c_{0}(L)$
evaluated on a sequence of minimizing loops based on $x$ whose lengths go to infinity. Looking closer
at the relationship between first and second Peierl's barriers we have (see for instance \cite{Fa1}):

\begin{lemma} The following assertions hold:
\begin{enumerate}
\item For every pair of conjugate functions $u^{-}$, $u^{+}$ we have
$$u^{-}(x) - u^{+}(x) \leq h(x,y) $$
for every $x,y \in M$.
\item $h(x,y) = \sup_{u^{-},u^{+}}\{u^{-}(y)-u^{+}(x) \}$ where the
supremum runs over all pair of conjugate pairs.
\end{enumerate}
\end{lemma}

\subsection{Static classes and uniqueness of the weak KAM solutions}

The dynamics of the set of global minimizers determines the uniqueness of the
solutions of the Hamilton-Jacobi equation. In order to be more precise about this assertion we need some definitions.
Let us recall that a semistatic curve
$\beta:[a,b]\longrightarrow M$ is an absolutely continuous curve such that

$$A_{L+c[0]}(\beta)= h(\beta(a),\beta(b)).$$

A static curve $\alpha:[a,b]\longrightarrow M$ is an absolutely continuous curve such that

$$A_{L+c[0]}(\beta)= -h(\beta(b),\beta(a)).$$

A static curve is always semistatic. The curve $\beta :I
\longrightarrow M$ is static, if and only if, it is static
restricted to any interval contained in $I$. Since $h(x,y) +
h(y,x)$ is the infimum of the action of the Lagrangian at the
critical level $c[0]$, such curves are minimizers of the action.
In the case of mechanical Lagrangians, static curves are the
maximum points of the potential, while semistatic curves are
projections of orbits of the Euler-Lagrange flow which tend to the
singularities. Observe that \cite{CI} the Aubry set is the set of
static curves. So in our case, the only static curve is the closed
geodesic $\gamma$ supporting the Mather measure.
\bigskip

Two points $\theta_{1}$, $\theta_{2}$ in $TM$ are in the same static class if
$$h(\pi(\theta_{1},\theta_{2})) + h(\pi(\theta_{2},\theta_{1})) = 0.$$

According to our assumptions, there is only one static class,
whose elements are the points of $\gamma$. We need the following result
from \cite{CI}:

\begin{lemma}
Suppose that the set of static classes is unique. Then there exist (up to
additive constant) a unique pair of conjugate solutions of the Hamilton-Jacobi
equation.
\end{lemma}

This yields

\begin{corollary}
Suppose that there exists just one static class. Then $h(x,x)=P(x,x)$.
\end{corollary}

\subsection{Peierl's barrier in terms of the differences between Busemann functions}

Now, we are in shape to prove the main result of the section. We are going to link
the Busemann functions of the metric $g$ in $\tilde{M}$ to the
second Peierl's barrier $P(x,x)$. For this purpose we shall prove
that Busemann functions are naturally related with the Peierl's
barrier $h(x,x)$, and then apply the above results. Throughout the subsection, $\gamma = \gamma_{\theta}$ will
be the geodesic in the statement of Theorem 1, namely, the support of the Mather measure, $\theta \in T_{1}M$
is the initial condition of $\gamma_{\theta}$. We choose a lift $\gamma_{\tilde{\theta}}$ of $\gamma_{\theta}$
in $\tilde{M}$, and a tubular neighborhood $\tilde{N}$ of $\gamma_{\tilde{\theta}}$ such that the covering
map $\Pi: \tilde{M} \longrightarrow M$ restricted to $\tilde{N}$ is a diffeomorphism into a tubular neighborhood $N(\gamma)$ of $\gamma$.

\begin{proposition} \label{barrier}
Let $(M,g)$ satisfy the assumptions of the main theorem. Then,
\begin{enumerate}
\item $h(x,x)= P(x,x) = b^{-\tilde{\theta}}(\tilde{x}) + b^{\tilde{\theta}}(\tilde{x})$ for every $x \in
N$ and $\tilde{x} \in \tilde{N}$ such that $\Pi(\tilde{x})= x$.
\item There exist positive constants $A,B$ such that
$$ A\inf_{p \notin \tilde{N}}(b^{-\theta}(p) + b^{\theta}(p)) \leq h(\Pi(p),\Pi(p)) \leq B\inf_{p\notin \tilde{N}}(b^{-\theta}(p) + b^{\theta}(p)).$$
\end{enumerate}
\end{proposition}

We prove the proposition in many steps.

\begin{lemma} \label{relativepotential}
Let $x \in N(\gamma)$, let $T>0$ be the minimum period of
$\gamma$, and let $\beta:[0,T_{n}]\longrightarrow N_{\gamma}$ be a
closed geodesic loop parametrized by arc length such that
\begin{enumerate}
\item $\beta(0)=x=\beta(T_{n})$,
\item $\beta$ is homotopic to $n[\gamma]\vert $.
\end{enumerate}
Then there exists a function $\delta(n)$, with $\lim_{n\rightarrow
+\infty}\delta(n)=0$, such that
$$\vert A_{L+c[0](L)}(\beta)- (b^{-\tilde{\theta}}(\tilde{x})+b^{\tilde{\theta}}(\tilde{x})) \vert
\leq \delta(n),$$ where $\tilde{x}\in \tilde{N}$ is any lift of $x$ in
$\tilde{N}$.
\end{lemma}

\begin{proof}
Let $\tilde{\beta}\subset \tilde{N}$ be a lift
of $\beta$, and let $\tilde{x}=\tilde{\beta}(0)$. Let
$T_{\gamma_{\tilde{\theta}}}$ be the covering translation
preserving $\gamma_{\tilde{\theta}}$. Then, the assumption implies
that
$$\tilde{\beta}(T_{n})=T_{\gamma_{\tilde{\theta}}}^{n}(\tilde{x}).$$
Since the curvature of $(M,g)$ is negative but at the points of
$\gamma$ we have that
\begin{enumerate}
\item The (unique) geodesic $[\tilde{p},T_{\gamma_{\tilde{\theta}}}^{n}(\tilde{p})]$
joining $\tilde{p}$ to $T_{\gamma_{\tilde{\theta}}}^{n}(\tilde{p})$ is contained in $\tilde{N}$
(because of the convexity of the metric),
\item The minimum distance from $\tilde{\beta}[0,T_{n}]$ to
$\gamma_{\tilde{\theta}}$ must converge to $0$ as $n\rightarrow
+\infty$. Namely, given $\epsilon>0$, there exist $n>0$
such that for every $\tilde{p}\in \tilde{N}$,
we have
$$\inf_{q \in [\tilde{p},T_{\gamma_{\tilde{\theta}}}^{n}(\tilde{p})]} d(q,\gamma_{\tilde{\theta}}) \leq \epsilon.$$
\end{enumerate}

Let $\beta_{x,+}[0,+\infty)\subset \tilde{N}$ be
the geodesic asymptotic to $\gamma_{\tilde{\theta}}$ with
$\beta_{x,+}(0)= \tilde{x}$. Let $\beta_{x,-}(-\infty,0]$ be the
geodesic asymptotic to $\gamma_{-\tilde{\theta}}$ with
$\beta_{x,-}(0)= T_{\gamma_{\tilde{\theta}}}^{n}(\tilde{x})$. Let
us denote by $[p,q]$ the geodesic joining the points $p$, $q$ in
$\tilde{M}$. We can assume without loss of generality that
$\gamma_{\tilde{\theta}}(0)$ is the closest point in
$\gamma_{\tilde{\theta}}$ to $\tilde{x}$. Since $\beta_{x,+}$,
$\beta_{x,-}$ are asymptotic to $\gamma_{\tilde{\theta}}$ there
exist $\mu(n)\rightarrow 0$ if $n\rightarrow +\infty$, and a
number $s_{n}>0$ satisfying
$$d(\beta_{x,+}(s_{n}),\beta_{x,-}(-s_{n})) \leq \mu(n).$$
Let us consider the broken geodesic $\tilde{\alpha}_{n}$ given by
$$ \tilde{\alpha}_{n}= \beta_{x,+}[0,s_{n}]\cup
[\beta_{x,+}(s_{n}),\beta_{x,-}(-s_{n})]\cup
\beta_{x,-}[-s_{n},0].$$

By the convexity of the metric, $\tilde{\alpha}_{n}$ is contained
in the region bounded by $\tilde{\beta}[0,T_{n}]$,
$\gamma_{\tilde{\theta}}$, and the geodesics
$[\tilde{x},\gamma_{\tilde{\theta}}(0)]$,
$[T_{\gamma_{\tilde{\theta}}}^{n}(x),\gamma_{\tilde{\theta}}(nT)]$.

Moreover, the length $l(\tilde{\alpha})$ of $\tilde{\alpha}$ is
$2s_{n}+d(\beta_{x,+}(s_{n}),\beta_{x,-}(-s_{n}))$, and
$s_{n}\rightarrow +\infty$ if $n\rightarrow +\infty$, so by the
definition of $\epsilon$ we have
$$\vert T_{n}-2s_{n} \vert \leq
\epsilon.$$

Let $a_{x}>0$ be defined by
$$\beta_{x,+}(a_{x,+})= H_{\tilde{\theta}}(0) \cap \beta_{x,+},$$
and let $a_{x,-} >0$ be defined by
$$\beta_{x,-}(-a_{x,-})= H_{-\tilde{\theta}}(nT) \cap \beta_{x,-}.$$

Let us still define $s_{n,+}>0$ by
$$ \beta_{x,+}(s_{n}) \in H_{\tilde{\theta}}(s_{n,+}),$$
and
$s_{n(b^{-\tilde{\theta}}(\tilde{x})+b^{\tilde{\theta}}(\tilde{x})),-}>0$
by
$$ \beta_{x,-}(-s_{n}) \in H_{-\tilde{\theta}}(-nT+s_{n,-}).$$

Notice that $\vert (s_{n,+}+s_{n,-})-nT \vert \leq \epsilon$. The
above definitions and the convexity of the metric yield
$$s_{n}=a_{x,+}+s_{n,+}= a_{x,-}+s_{n,-}.$$
And observe that
$$a_{x,+}=
-b^{\tilde{\theta}}(\beta_{x,+}(a_{x,+}))+b^{\tilde{\theta}}(\tilde{x}),$$
$$a_{x,-}=
-b^{-\tilde{\theta}}(\beta_{x,-}(-a_{x,-}))+b^{-\tilde{\theta}}(T_{\gamma_{\tilde{\theta}}}^{n}(\tilde{x})).$$
Since by definition,

\begin{enumerate}
\item $b^{\tilde{\theta}}(\beta_{x,+}(a_{x,+}))=0,$
\item $b^{-\tilde{\theta}}(\beta_{x,-}(-a_{x,-}))=+nT,$
\item $b^{-\tilde{\theta}}(T_{\gamma_{\tilde{\theta}}}^{n}(\tilde{x}))=b^{-\tilde{\theta}}(\tilde{x})
- nT,$
\end{enumerate}
we have
$$a_{x,+}+a_{x,-} =(b^{-\tilde{\theta}}(\tilde{x})+b^{\tilde{\theta}}(\tilde{x})).$$

Therefore, the length $l(\tilde{\alpha})=2s_{n}+
d(\beta_{x,+}(s_{n}),\beta_{x,-}(-s_{n}))$ satisfies
\begin{eqnarray*}
 2s_{n}  & = &
((a_{x,+}+a_{x,-})+s_{n,+}+s_{n,-} \\
& = & (s_{n,+}+s_{n-}) +
(b^{-\tilde{\theta}}(\tilde{x})+b^{\tilde{\theta}}(\tilde{x}))
\end{eqnarray*}
which yields \begin{eqnarray*} \vert (2s_{n} - nT ) -
(b^{-\tilde{\theta}}(\tilde{x})+b^{\tilde{\theta}}(\tilde{x}))
\vert \\
& = & \vert (s_{n,+}+s_{n,-}) -nT \vert \leq \epsilon .
\end{eqnarray*}

Hence, the length $T_{n}$ of $\tilde{\beta}$ satisfies
$$ \vert (T_{n}-nT) -(b^{-\tilde{\theta}}(\tilde{x})+b^{\tilde{\theta}}(\tilde{x}))  \vert
\leq 2\epsilon. $$

The same estimate holds for the curve $\beta$ that is homotopic to
$n[\gamma]$. To calculate the action of $L(p,v)+c[0](L)=
\frac{1}{2}(v,v) -\omega_{p}(v)+c[0](L)$ in $\beta$, take without
loss of generality $\parallel \gamma \parallel_{st}=1=T$, and
observe that
$$\int_{\beta}\omega = \int_{\gamma} \omega = n.$$
So we have
$$ A_{L+c[0](L)}(\beta)  =  \frac{1}{2}l(\beta) - n  +\frac{1}{2}l(\beta)  = l(\beta) -
n.$$ Thus,

$$\vert A_{L+c[0](L)}(\beta) - ((b^{-\tilde{\theta}}(\tilde{x})+b^{\tilde{\theta}}(\tilde{x}))
\vert \leq 2\epsilon.$$

As $n \rightarrow +\infty$, we can take $\epsilon$ arbitrarily
small, and this implies the Lemma.
\end{proof}

\begin{lemma} \label{tubularminimizers}
There exists a tubular neighborhood $V\subset N(\gamma)$ of $\gamma$ with the following property: Let
$x \in V$, and let $\gamma_{n}$ be the sequence defined in Lemma \ref{PeierlAubry}
of closed geodesic loops based on $x$ such that $h(x,x) = \lim_{n\rightarrow \infty}(A_{L+c_{0}(L)}(\gamma_{n}))$.
Then there exists $m(x)>0$ such that $\gamma_{n} \subset N(\gamma)$ for every $n \geq m(x)$.
\end{lemma}

\begin{proof}
This is will follow from the fact that the Aubry set and the Mather set are equal to
the closed geodesic $\gamma$. Indeed, suppose that the statement is not true. Then we can choose a
sequence $x_{i}$ of points converging to $\gamma$ with the following property:

Let $\gamma_{n}^{i}$ be a sequence of closed geodesic loops based on $x_{i}$
with $h(x_{i},x_{i}) = \lim_{n\rightarrow \infty}(A_{L+c_{0}(L)}(\gamma_{n}^{i}))$ (Lemma \ref{PeierlAubry}.
Then there exist a subsequence $\gamma_{n_{i}}^{i}$, $n_{i} \rightarrow \infty$ if $i\rightarrow \infty$,
such that
\begin{enumerate}
\item $\gamma_{n_{i}}^{i}$ is not contained in $N(\gamma)$ for every $i$,
\item $ \lim_{i\rightarrow \infty}\vert h(x_{i},x_{i}) - (A_{L+c_{0}(L)}(\gamma_{n_{i}}^{i}))\vert = 0$.
\end{enumerate}
Since the loops $\gamma_{n_{i}}^{i}$ are minimizers based at $x_{i}$ where $h(x_{i},x_{i})$ tends to zero with
$i\rightarrow \infty$, and their domains tend to ${\mathbb R}$, there exists a
subsequence of them converging to a global minimizer $\beta$. The geodesic $\beta$ would have a point outside
$N(\gamma)$ and, by the continuity of $h(x,x)$, $\beta$ would also contain a point $p$ with $h(q,q)=0$.
By Lemma \ref{PeierlAubry} the point $q$ is in the Aubry set which coincides with the Mather set: the closed
geodesic $\gamma$. This is clearly a contradiction.
\end{proof}

\begin{lemma} \label{uniquemin}
There exists a constant $C>0$ such that if $x \notin N(\gamma)$,
then $h(x,x) \geq C \inf_{p\notin
\tilde{N}}\{b^{-\tilde{\theta}}(p) +
b^{\tilde{\theta}}(p)\}$.
\end{lemma}

\begin{proof}
Since the function $h(x,x)$ is zero just at the points of the Aubry set (the closed geodesic
$\gamma$) and is continuous, this implies that outside the tubular neighborhood $N(\gamma)$
it must be strictly positive. In the same way, the function $ b^{-\tilde{\theta}}(p) +
b^{\tilde{\theta}}(p)$ is strictly positive outside any tubular
neighborhood of $\gamma_{\tilde{\theta}}$. The periodicity of the Busemann function of $\gamma_{\tilde{\theta}}$,
and the convexity of the metric, imply that the function $d(p)=b^{-\tilde{\theta}}(p) +
b^{\tilde{\theta}}(p)$ is convex in $\tilde{M}$ and attains its minimum value outside $\tilde{N}$ at
the boundary of $\tilde{N}$. This minimum value is positive, so $h(x,x)$ and $d(\tilde{x})$ have analogous
behaviour and the comparison stated in the lemma follows.
\end{proof}
\bigskip

{\bf Proof of Proposition \ref{barrier}}

By Lemma \ref{tubularminimizers}, if we want to estimate $h(x,x)$
at points $x \in V\subset N(\gamma)$ it is enough to consider the loops based at $x$
contained in $N(\gamma)$, where $V$ is the tubular neighborhood of lemma \ref{tubularminimizers}.
Then, lemma \ref{relativepotential} combined with this observation yields item (1). Item (2) follows
from Lemma \ref{uniquemin}.
\bigskip

{\bf Remark}
\bigskip

Proposition \ref{barrier} has some interesting consequences regarding the regularity of the Peierl's barrier
that is closely related to the regularity of the solutions of the
Hamilton-Jacobi equation. The proposition tells us that the function $h(x,x)$ is, in a tubular neighborhood of
$\gamma$, as regular as the Busemann functions of any lift of the geodesic $\gamma$ in $\tilde{M}$.
In particular, if the manifold $(M,g)$ has non-positive curvature, the regularity of the Busemann functions
and Peierl's barrier is $C^{3}$ \cite{E}. The local regularity of Peierl's barrier and subsolutions of the
Hamilton-Jacobi equation close to the support of the Mather measure is quite exceptional, it holds
for instance when the support is a finite collection of closed hyperbolic orbits \cite{Bernard}.

\section{Proof of the main Theorem}
\subsection{Surfaces of revolution as models}

The main idea to show Theorem 1 is to apply our estimates of the
Peierl's barrier in terms of Busemann functions in order to obtain explicit
bounds for the deviation function for surfaces of revolution with
nonpositive curvature. Then, to get bounds for the deviation
function for the surfaces in the statement of Theorem 1 we use
comparison theorems to get information about the asymptotic
behavior of geodesics.


\bigskip

We choose polynomial functions as test functions to generate
annulus of revolution: If $(x,y,z)$ are the cartesian coordinates
in ${\mathbb R}^ {3}$ let us consider the curves
\[ r(z)= (a+z^{2+k},0,z), \]
where $k \in {\mathbb N}$. (The exponent $2+k$ grants the
existence of Gaussian curvature). We would like to point out that

\begin{itemize}
 \item If the annulus is $C^{\infty}$ then $k$ is even.
\item We can also consider $k \in {\mathbb R}^{+}$, but in this case the annulus might
not be $C^{\infty}$.
\end{itemize}

\subsection{Estimates for Busemann functions in surfaces of revolution}

To estimate the Busemann functions we use the {\bf Clairaut
equation}
\[ r(\gamma(t))cos(\theta(\gamma(t))) = c \]
where $\gamma(t)$ is a geodesic parametrized by arc length, $r(p)$
is the distance from $p$ in the annulus to the revolution axis
$\overrightarrow{z}$, and $\theta(\gamma(t))$ is the angle formed
by the geodesic $\gamma$ and the parallel $z= z(\gamma(t))$ (in
our case, $c=a$).
\bigskip

It follows from a result by P. Eberlein that the Busemann functions in nonpositive curvature are $C^{2}$, so we
can write them as functions of $z$ in the following way (in fact,
by the symmetries of the annulus, it is enough to consider the
generating curve):
\begin{eqnarray*}
u^{+}(z)+u^{-}(z) & = & u^{+}(0)+u^{-}(0) + \int_{0}^ {z}<\nabla u^{+}(t)+ \nabla u^{-}(t),\frac{\delta}{\delta t}>dt \\
 & = & 2\int_{0}^{z} sin(\theta(t)) dt \\
 & = & 2\int_{0}^{z} \sqrt{1-(\frac{a}{r(t)})^{2}}dt \\
& = & 2\int_{0}^{z}
\frac{t^{1+\frac{k}{2}}}{a+t^{2+k}}\sqrt{2a+t{2+k}}dt
\end{eqnarray*}
from which follows the estimate $h(z,z) = Cz^{2+\frac{k}{2}} +
o(z^{3+\frac{k}{2}})$.

\subsection{Comparison theory and estimates for Busemann functions
close to non-hyperbolic geodesics}

The purpose of the subsection is to apply the estimate of the previous subsection to
surfaces of nonpositive curvature $(M,g)$ with $K\leq 0$ like in Theorem 1. Namely,
\begin{enumerate}
 \item There is a closed geodesic $\gamma$ of period $T$ where $K\equiv  0$
whose orbit supports the Mather measure of
$L(p,v)=\frac{1}{2}g_{p}(v,v)-\omega_{p}(v)$.
\item $K<0$ in the complement of $\gamma$.
\item There exists $m>0$ such that for every geodesic
$\beta: (-\epsilon, -\epsilon) \to M$
 perpendicular to $\gamma$ at $\beta(0)=\beta \cap \gamma$
we have that $m$ is the least integer where
 $\frac{\partial^{m}}{dt^{m}} K(\beta(t))\vert_{t=0}\neq 0.$
\end{enumerate}

Such surfaces might not have an annulus of revolution containing the geodesic $\gamma$ as a waist,
so the estimates in the previous subsection might not apply immediately. We shall show that in fact,
the sum of Busemann functions $u^{+}(z)+u^{-}(z)$ is a function of the {\bf angle} formed by the two geodesics
through $z \in \tilde{M}$ which are respectively, forward and backward asymptotic to a lift of $\gamma$.
Then, using CAT comparison theorems we shall be able to compare this function with its counterpart
in the model surface considered in the previous subsection. This will yield a comparison of the sums
of Busemann functions in both surfaces.

To begin with, let us consider a Fermi coordinate system $\Phi:S^{1}_{T}\times
(-\epsilon, \epsilon) \to M$, where $S^{1}_{T}$ is a
circle of length $T$ parametrized by arc length, such that
$\Phi(t,0)=\gamma(t)$ for every $t \in [0,T]$. The curves
$\Phi_{t}(s)=\Phi(t,s)$, $s\in (-\epsilon, \epsilon)$, are
geodesics which are perpendicular to $\gamma$.

Let $A_{b}$ be the annulus of revolution generated by rotating the
curve
\[ r_{b}(z)= (a+bz^{2+k},0,z), \]

around the vertical axis in ${\mathbb R}^{3}$. The meridians of
the annulus of revolution are geodesics, and we can use the same
map $\Phi$ to parametrize a subset of the annulus containing the
waist $\gamma_{0}(t)=R_{t}(r_{b}(0))$, where $R_{t}$ is the
rotation around the $z$-axis of angle $t$ (We can take $a>0$ such
that the length of $\gamma_{0}$ is $T$). Item (3) above implies
that

\begin{lemma}
Given $\sigma>0$ there exist a tubular neighborhood
$V_{\gamma}$ and constants $b_{1}<b_{2}$ with $\vert b_{2}-b_{1}
\vert <\sigma$, such that
$$ K_{b_{1}}(\Phi(p)) \leq K(\Phi(p)) \leq K_{b_{2}}(\Phi(p)) ,$$
where $K_{b}$ is the Gaussian curvature of $A_{b}$.
\end{lemma}

\begin{proof}
From the assumptions on the surface $(M,g)$, the curvature $K$ when
restricted to a geodesic $\beta: (-\epsilon, -\epsilon) \longrightarrow M$
perpendicular to $\gamma$ at $\beta(0)=\beta \cap \gamma$
satisfies
$$\frac{\partial^{m}}{ds^{m}} K(\beta(s))\vert_{s=0}\neq 0.$$
where $m$ is the least integer $k>0$ where the $k$-th derivative of
$K(\beta(s))$ is different from 0. Also from the assumptions, the number
$m$ is  the same for every geodesic $\beta$ as before. So we can find
two functions of the form $f_{1}(s)= a+b_{1}s^{2+m}$, $f_{2}(s)= a+b_{2}s^{2+m}$,
such that
\begin{enumerate}
\item The annulus of revolution $A_{b_{1}}$, $A_{b_{2}}$ generated by
the graphs of $f_{1}$, $f_{2}$ respectively, have waists $\gamma_{1}$, $\gamma_{2}$
of period $2\pi a = T$, where $T$ is the minimum positive period of $\gamma$.
\item $-\frac{\partial^{m+2}}{ds^{m+2}}f_{2}(0) < K(\beta(0)) < -\frac{\partial^{m+2}}{ds^{m+2}}f_{1}(0) < 0$,
for every geodesic $\beta: (-\epsilon, \epsilon)\longrightarrow M$ perpendicular to $\gamma$.
\end{enumerate}
By the theory of surfaces of revolution, the curvatures of $A_{b_{i}}$ at
their waists $\gamma_{i}$ are just the opposites of the second derivatives
of the $f_{i}$ at $t=0$. By continuity, there exists a tubular neighborhood
of $\gamma$ in the parametrization $\Phi$ satisfying the assertion of the lemma.
\end{proof}

The estimates for Busemann functions of $A_{b}$ are completely analogous (up to
multiplication by constants) to the estimates for $A_{1}$ showed
in the previous subsection. Now, we state a version of the comparison theorem for angles of
geodesic triangles, suited for our purposes. It is based in the well known comparison theorems for geodesic triangles
due to A. D. Alexandrov and V. A. Toponogov, which compare the geometry of geodesic triangles in a
Riemannian manifold $(M,g)$ with the geometry of geodesic triangles in spaces of constant curvature
(the book of Cheeger and Ebin \cite{CE} is a great reference for the subject).
We shall use the surfaces of revolution $A_{b}$ described in the previous subsection instead of manifolds
with constant curvature as comparison spaces. Since this is not the usual version of CAT theorems, we
give a proof of the result at the end of the section.

\begin{theorem} \label{CAT} (Comparison theorem for angles)
Let $S= {\mathbb R} \times (-\epsilon, \epsilon)$ be a strip, and let $S_{1}=(S,g_{1})$, $S_{2}=(S,g_{2})$
be two $C^{\infty}$ Riemannian metrics in $S$ such that
\begin{enumerate}
\item The curvatures $K_{1}$, $K_{2}$ of $S_{1}$, $S_{2}$ respectively, satisfy
$K_{2}(t,s) \leq K_{1}(t,s)\leq 0$ for every $(t,s) \in S$.
\item The curves $c(t)= (t,0)$, $t \in {\mathbb R}$, and $\sigma_{t}(s)=(t,s)$,
$s \in (-\epsilon, \epsilon)$, are geodesics in $S_{1}$ and $S_{2}$ for every $t \in {\mathbb R}$.
\item $c(t)$ has unit speed in $S_{1}$ and $S_{2}$, $\sigma_{t}$ has unit speed in $S_{1}$
for every $ t \in {\mathbb R}$, and $\sigma_{0}$ is perpendicular to
$c(t)$ in $S_{1}$ and $S_{2}$.
\end{enumerate}
Let $\Delta_{i}(t,s)$ be the $S_{i}$-geodesic triangles whose common vertices are $c(0), \sigma_{0}(s), c(t)$.
Let $[\sigma_{0}(s), c(t)]_{i}$ be the $S_{i}$-geodesic joining $\sigma_{0}(s)$ and $ c(t)$. Let
$\alpha_{i}(t,s)$ be the $S_{i}$-angle formed by $\sigma_{0}'(s)$ and $[\sigma_{0}(s), c(t)]_{i}$ at
the point $\sigma_{0}(s)$. Then
\begin{enumerate}
\item $\Delta_{2}(t,s) \subset \Delta_{1}(t,s)$ for every $t,s$.
\item We have that $\alpha_{2}(t,s) \leq \alpha_{1}(t,s)$. Moreover, they coincide if and only if
$\Delta_{1}(t,s)=\Delta_{2}(t,s)$ and $K_{1}(p)=K_{2}(p)$ for every $p$ inside $\Delta_{i}(t,s)$.
\end{enumerate}
\end{theorem}

Now, we can apply Theorem \ref{CAT} for angles to bound from above
and from below the distance from $\gamma$ to its asymptotes. Namely,
let us take $\beta$ like before, with $\beta(0)=\gamma(t_{0})$, and
we consider $\beta(a)=\Phi(p)$, $a \in [0,\epsilon)$, $t \neq
t_{0}$. Let $\Delta(p,t)$ be the geodesic triangle in $(M,g)$ with
vertices $\Phi(p)$, $\gamma(t_{0})=\Phi(p_{0})$,
$\gamma(t)=\Phi(p_{t})$; and let $\Delta_{b}(p,t)$ be the geodesic
triangle in $A_{b}$ with the same vertices (let us remind that we
are parametrizing a tubular neighborhood of $\gamma$ and a
subannulus of $A_{b}$ by the map $\Phi$). Then by Theorem \ref{CAT},
if $K_{b}(\Phi(x)) \leq K(\Phi(x))$ for every $x$ in $(-\epsilon,
-\epsilon)\times [0,T]$, we have that $\Delta_{b}(p,t) \subset
\Delta(p,t)$ for every $t \in \mathbb{ R}$. So if we let
$t\rightarrow \infty$, the same property remains true for the ideal
triangles in both annuli with vertices $\Phi(t_{0})$, $\Phi(p)$.

Therefore, the angles formed by $\beta$ and the geodesic
$\gamma_{p}$ through $\Phi(p)$ that is asymptotic to $\gamma$ is at least
the angle formed by $\beta$ and the geodesic $\gamma_{b,p}$ in $A_{b}$
asymptotic to $\gamma_{b}$, whenever $K_{b}(\Phi(x))\leq K(\Phi(x))$ for every
$x$.
\bigskip

This application of comparison theory leads us to a
comparison between Busemann functions in the annuli $A_{b}$, $V_{\gamma}$.

Indeed, let us recall that the formula for the sum of the Busemann
functions in the previous subsection was
$$u^{+}(s)+u^{-}(s) =  u^{+}(0)+u^{-}(0) + \int_{0}^ {s}<\nabla u^{+}(\rho)+ \nabla u^{-}(\rho),\frac{\delta}{\delta
\rho}>d\rho= $$
$$\int_{0}^ {s}<\nabla u^{+}(\rho)+ \nabla
u^{-}(\rho),\frac{\delta}{\delta \rho}>d\rho. $$

The same formula essentially holds for the tubular neighborhood
$V_{\gamma}$ in $(M,g)$, taking $\frac{\delta}{\delta s}$ as the
Fermi coordinate vector field tangent to the geodesics
$\Phi_{t}(s)$ (we assume that the pair
$\gamma'(t),\frac{\delta}{\delta s}$ has the canonical
orientation). Namely, given $t \in [0,T]$ we have
$$u^{+}(\Phi(t,s))+u^{-}(\Phi(t,s)) =  \int_{0}^ {s}<\nabla u^{+}(\Phi(t,\rho))+ \nabla
u^{-}(\Phi(t,\rho)),\frac{\delta}{\delta \rho}>d\rho. $$

So we get
$$u^{+}(\Phi(t,s))+u^{-}(\Phi(t,s)) = \int_{0}^ {s}(\cos(\theta^{+}(t,\rho))) + \cos(\theta^{-}(t,\rho))))dt, $$
where $\theta^{+}(t,s)$ is the angle formed by the forward
asymptote of $\gamma$ containing $\Phi(t,s)$ with the vector
$\frac{\delta}{\delta s}(\Phi(t,s))$ (respectively,
$\theta^{-}(t,s)$ is the angle formed by the backward asymptote of
$\gamma$ and $\frac{\delta}{\delta s}(\Phi(t,s))$).
\bigskip

Let $u_{b}^{+}$, $u_{b}^{-}$ be the Busemann functions in $A_{b}$,
and let $\theta_{b}^{+}(t,s)$, $\theta_{b}^{-}(t,s)$ be the angles formed by
the forward (resp. backward) asymptotes of $\gamma_{0}$ and the
vector field $\frac{\delta}{\delta s}$ at the point $\Phi(t,s)$. By Theorem \ref{CAT} we have
that if $K_{b}(\Phi(t,s))< K(\Phi(t,s))$ then
$$\theta_{b}^{+}(t,s) < \theta^{+}(t,s),$$ and
$$\theta_{b}^{-}(t,s) < \theta^{+}(t,s).$$
Moreover, if $K(\Phi(t,s)) < K_{b}(\Phi(t,s))$ then
$$\theta^{+}(t,s) < \theta_{b}^{+}(t,s),$$
$$\theta^{+}(t,s) < \theta_{b}^{+}(t,s).$$
Replacing these inequalities in the above integrals we get lower
and upper bounds for the sum of the Busemann functions in
$V_{\gamma}$ in terms of the formulae obtained in the previous
subsection. This finishes the proof of Theorem 1.
\bigskip

{\bf Proof of Theorem \ref{CAT}}
\bigskip

Roughly speaking, the less curved is the space the more convex is the distance between geodesics
and narrower are the angles of geodesic triangles. Such features of comparison geometry are well known
since the famous CAT comparison theorems. We shall give an elementary proof using
Gauss-Bonnet theorem for geodesic polygons for the sake of completeness.

First some general remarks about angles and lengths. Angles are conformal invariants, and since
two metrics in an open, simply connected set of the plane are conformally equivalent we have
that the angles with respect to $S_{1}$, $S_{2}$ formed by two vectors $v,w$ tangent to a point in $S$
are the same. Moreover, the $S_{2}$-length of a curve $c\subset S$ is at least its $S_{1}$-length,
because the decrease of the curvature increases the norm of vectors and length.

The Gauss-Bonnet theorem for a geodesic triangle $\Delta$ with inner angles $a, b ,c$ tells us that
$$ (a+b+c) = \int_{\Delta}K   + \pi ,  $$
and the Gauss-Bonnet theorem for geodesic quadrilaterals $\Box$ with inner angles $a, b, c, d$ is
$$ (a+b+c+d) = \int_{\Box}K   + 2\pi .  $$
Now, let us consider $t = -T$, for $T>0$, let us denote $[\sigma_{0}(s), c(-T)]_{i}= \gamma_{i}^{s}$,
and notice that $\gamma_{i}^{s}$ can be parametrized in terms of the parameter $t$ since these geodesics are everywhere
transversal to the vertical geodesics $\sigma_{t}$. So let $\gamma_{i}^{s}(t) = \gamma_{i}^{s} \cap \sigma_{t}$.
The triangles $\Delta_{i}(-T,s)$ have two sides in common, the geodesics $c(t), t\in [-T,0]$, and $\sigma_{0}(r)$,
$r \in [0,s]$. The geodesics $\gamma_{i}^{s}$ are the other sides of $\Delta_{i}(t,s)$, and we claim
\bigskip

{\bf Claim 1:} $\gamma_{2}^{s}$ is contained in $\Delta_{1}(-T,s)$.
\bigskip

For suppose that $\gamma_{2}^{s}$ is not contained in $\Delta_{1}(-T,s)$. Then there exist $t_{0} \in (-T,0]$,
$s_{0} \in (0,s)$ such that
\begin{enumerate}
\item $\sigma_{t_{0}}(s_{0}) \in \gamma_{1}^{s}\cap \gamma_{2}^{s}$,
\item the angle formed by $\frac{d}{dt}\gamma_{2}^{s}$ and $\sigma_{t_{0}}'(s_{0})$ is at least
the angle formed by $\frac{d}{dt}\gamma_{1}^{s}$ and $\sigma_{t_{0}}'(s_{0})$.
\end{enumerate}
Since $\gamma_{i}^{s}(-T)= c(-T)$ there exists a minimum value $r_{0} <t_{0}$ such that
\begin{enumerate}
\item $\gamma_{2}^{s}(t)$ is not in $\Delta_{1}(-T,s)$ for every $t \in (r_{0},t_{0})$,
\item $\gamma_{2}^{s}(r_{0}) = \gamma_{1}^{s}(r_{0})$.
\end{enumerate}
Then, there exists $z >s$ such that the geodesic $\gamma_{1}^{z}$ is tangent to $\gamma_{2}^{s}$ at
some point $\gamma_{1}^{z}(t_{1})$, where $t_{1} \in (r_{0},t_{0})$. Let $s_{1}<s$ be such that
$\sigma_{t_{1}}(s_{1}) = \gamma_{1}^{s}(t_{1})$. Let $z_{1}>s_{1}$ be such that
$\gamma_{2}^{s}(t_{1})= \sigma_{t_{1}}(z_{1})$.
Let us consider the geodesic quadrilaterals $\Box_{1}$, $\Box_{2}$
with the following geodesic sides and angles:
\begin{enumerate}
\item The sides of $\Box_{1}$ are the $S_{1}$-geodesics
$c(t), t \in [t_{1},t_{0}]$, $\sigma_{t_{1}}([0,s_{1}])$,
$\sigma_{t_{0}}([0,s_{0}])$ and $\gamma_{1}^{s}([t_{1},t_{0}])$. The inner angles are
$a_{1}$, the angle formed by $c(t)$ and $\sigma_{t_{1}}$; $a_{2}$, the angle formed
by $c(t)$ and $\sigma_{t_{0}}$; $a_{3}$, the angle formed by $\sigma_{t_{1}}$ and
$\gamma_{1}^{s}$; and $a_{4}$ the angle formed by $\sigma_{t_{0}}$ and $\gamma_{1}^{s}$.
\item The sides of $\Box_{2}$ are the $S_{2}$-geodesics $c(t), t \in [t_{1},t_{0}]$, $\sigma_{t_{1}}([0,z_{1}])$,
$\sigma_{t_{0}}([0,s_{0}])$ and $\gamma_{2}^{s}([t_{1},t_{0}])$. The inner angles are $b_{j}$, $j=1,2,3,4$, where
$b_{1}$ is the angle formed by $c(t)$ and $\sigma_{t_{1}}$; $b_{2}$, the angle formed
by $c(t)$ and $\sigma_{t_{0}}$; $b_{3}$, the angle formed by $\sigma_{t_{1}}$ and
$\gamma_{2}^{s}$; and $b_{4}$ the angle formed by $\sigma_{t_{0}}$ and $\gamma_{2}^{s}$.
\end{enumerate}
Observe that $\Box_{1} \subset \Box_{2}$, and that $a_{1}=a_{2}=b_{1}=b_{2}=\frac{\pi}{2}$.
So the sum of the inner angles of $\Box_{1}$ is
$$\pi + (a_{3} + a_{4}).$$
The sum of the inner angles of $\Box_{2}$ is
$$ \pi + (b_{3}+b_{4}).$$

Moreover, since $\gamma_{1}^{z}$ is tangent to $\gamma_{2}^{s}$ at
$\gamma_{1}^{z}(t_{1}) = \sigma_{t_{1}}(z_{1})$, and $z>s$, we have that
$$b_{3} \geq a_{3}. $$
By Gauss-Bonnet we get that
$$ \sum_{j=1}^{4}b_{j} - \sum_{j=1}^{4}a_{j} = (b_{3} - a_{3}) + (b_{4} - a_{4}) = \int_{\Box_{2}}K_{2} - \int_{\Box_{1}} K_{1} .$$
This implies that
\begin{eqnarray*}
 (b_{4} - a_{4}) & = & \int_{\Box_{2}}K_{2} - \int_{\Box_{1}} K_{1} - (b_{3} - a_{3})\\
 & \leq & \int_{\Box_{2}}K_{2} - \int_{\Box_{1}} K_{1}.
\end{eqnarray*}

Since, by assumption, the tangent vector of $\gamma_{2}^{s}$ at $\gamma_{2}^{s}(t_{0})$ points into $\Delta_{1}(t,s)$,
we know that $b_{4} \geq  a_{4}$, or $b_{4}-a_{4} \geq 0$.
So we would have that the difference of the integrals in the right is positive. However,
\bigskip

{\bf Claim 2:} The number $\int_{\Box_{2}}K_{2} - \int_{\Box_{1}} K_{1}$ is negative.
\bigskip

In fact, we have
$$\int_{\Box_{2}}K_{2} - \int_{\Box_{1}} K_{1} = \int_{\Box_{1}}K_{2} - \int_{\Box_{1}} K_{1} - \int_{\Box_{2}-\Box_{1}} K_{2} .$$
So
$$\int_{\Box_{2}}K_{2} - \int_{\Box_{1}} K_{1} \leq  \int_{\Box_{1}}K_{2} - \int_{\Box_{1}} K_{1}.$$
The integral of $K_{i}$ is calculated in terms of the area form of $S_{i}$. If we put both integrals in terms
of $dtds$, the coordinates of the strip $S$, we get
$$ \int_{\Box_{1}} K_{1} = \int_{\Box_{1}} J_{1}(t,s)K_{1}(t,s)dtds ,$$
$$ \int_{\Box_{1}} K_{2} = \int_{\Box_{1}} J_{2}(t,s)K_{2}(t,s)dtds ,$$
where $J_{i}(t,s)$ is the Jacobian of the coordinate change of the $S_{i}$-area form. Since $K_{2}(t,s) \leq K_{1}(t,s)$
the Jacobians satisfy $J_{2}(t,s)\geq J_{1}(t,s)$. Thus, $\int_{\Box_{1}}K_{2} - \int_{\Box_{1}} K_{1} \leq 0$ and the claim
is proved.
\bigskip

The Claim contradicts $b_{4}-a_{4}\geq 0$, which was a consequence of the assuming that $\gamma_{2}^{s}$ is
not contained in the triangle $\Delta_{1}(-T,s)$. This shows Claim 1.

Theorem \ref{CAT} follows from Claim 1 (that is item (1) in the statement), since $\Delta_{2}(-T,s) \subset \Delta_{1}(-T,s)$
obviously implies that the inner angles of $\Delta_{2}(-T,s)$ are bounded above by the inner angles of $\Delta_{1}(-T,s)$.
Gauss-Bonnet implies metric rigidity if the two triangles coincide. The argument for $T<0$ is completely analogous.


\begin{thebibliography}{99}

   \vspace{0.6cm}



 \bibitem {A1}     N. Anantharaman.  \emph{ On the zero-temperature or vanishing viscosity
limit for certain Markov processes arising from Lagrangian dynamics},    J. Eur. Math. Soc.,
 6 (2): , (2004), 207--276.

\bibitem{A2}     N. Anantharaman.   \emph{ Counting geodesics which are optimal in homology},    Erg. Theo. and Dyn. Syst.,
   23 (2): , (2003), 353--388.

\bibitem{A3}     N. Anantharaman.
\emph{ Entropie et localisation des fonctions propres},    ENS-Lyon,
(2006),
http://www.math.polytechnique.fr/$~$nalini/HDR.pdf

\bibitem{AIPS}      N. Anantharaman, R. Iturriaga, P.  Padilla, H. S烱chez-Morgado,
\emph{ Physical solutions of the Hamilton-Jacobi equation},
Disc. Contin. Dyn. Syst. Ser. B 5 (3) (2005), 513--528.

\bibitem{BGS}   W. Ballman, M. Gromov and  V. Schroeder,
\emph{ Manifolds of Non-positive Curvature}, , (1985), Birkhausser.

\bibitem{Bernard}     P. Bernard, \emph{ Smooth critical sub-solutions of the Hamilton-Jacobi equation},
Math. Res. Lett. 14 (2007), n. 3, 503--511.

\bibitem{Bus}   H. Busemann,
\emph{ The geometry of geodesics}, (1955), Academic Press,

\bibitem{Car} 
M. J. D. Carneiro,    \emph{ On minimizing measures of the action of
autonomous Lagrangians},    Nonlinearity 8 (6), (1995), 1077--1085

\bibitem{CE} Cheeger, J., Ebin, D. G.
\emph{ Comparison theorems in Riemannian geometry.}
North-Holland Mathematical Library, Vol. 9. North-Holland Publishing Co.,
Amsterdam-Oxford; American-Elsevier Publishing Co., Inc., (1975), New York.

\bibitem{Co1} 
G. Contreras,   \emph{ Action potential and weak KAM solutions},    Calc.
Var. Partial Differential Equations 13  (4)  (2001), 427--458.

\bibitem{CR} 
G. Contreras and R. Ruggiero,   \emph{  Non-hyperbolic surfaces having all ideal
triangles of finite area},    Bul. Braz. Math. Soc. 28, (1)  (1997),  43--71.



 \bibitem{CI}    G Contreras and R. Iturriaga.  \emph{ Global Minimizers of
 Autonomous    Lagrangians},     (2004),  To appear.


\bibitem{CIPP} 
G. Contreras, R. Iturriaga, G.P. Paternain
and M. Paternain,    \emph{ Lagrangian Graphs, minimizing measures and Ma\~n\'e's critical values},    Geom. Funct. Anal.
Vol. 8  (1998), 788--809.


 \bibitem{DZ}   A. Dembo and O. Zeitouni,  \emph{  Large Deviations Techniques and Applications},    (1998), Springer Verlag.

\bibitem{E} 
P. Eberlein,   \emph{ Manifolds of nonpositive curvature},    Global differential geometry, 223--258, MAA Stud. Math., 27 (1989), Math. Assoc. America, Washington, DC

 \bibitem{Du}   R. Durrett.
\emph{ Stochastic Calculus: A practical Introduction},    (1996), CRC-Press,



\bibitem{Ev1}    L. C. Evans,    \emph{ Towards a Quantum Analog of Weak KAM Theory}, Commun. in Math. Phys., (2004), 244, 311--334.



 \bibitem{Fa}   A. Fathi.    \emph{ Weak KAM Theorem and
     Lagrangian Dynamics},    (2004), To appear.


\bibitem{Fa1}     A. Fathi,
\emph{ Solutions KAM faibles conjuguees et barrieres de Peierls},    C.
R. Acad. Sci. Paris, Ser. I 235 (1997), 649--652.

\bibitem{FaSi}      A. Fathi, A. Siconolfi,
\emph{ Existence of $C^{1}$ critical subsolutions of the Hamilton-Jacobi equation},
Invent. Math. 155, n. 2, (2004), 363--388.

\bibitem{GN} 
M. Gerber and  V. Nitica,   \emph{ Holder exponents of horocycle
foliations on surfaces},   Ergodic Theory Dynam. Systems 19, no. 5  (1999),
1247--1254.



 \bibitem{GV}     D. A. Gomes and  E. Valdinoci,   \emph{ Entropy
     Penalization Methods for Hamilton-Jacobi Equations},   Adv. Math. 215, No. 1  (2007), 94--152.

 \bibitem{GLM}    D. A.  Gomes, A. O. Lopes and J. Mohr,   \emph{ The Mather measure and a Large Deviation Principle for
the Entropy Penalized Method},   To
appear in Communications in Contemporary Mathematics.


 \bibitem{GLM2}   D. A.  Gomes, A. O. Lopes and J. Mohr,   \emph{
Wigner measures and  the  semi-classical limit to the
Aubry-Mather measure},     (2009), To appear preprint

\bibitem{LRR} 
A. O. Lopes, V. Rosas and R. Ruggiero,   \emph{ Cohomology and
Subcohomology for expansive geodesic flows},  Discrete and
Continous Dynamical Systems,  Volume: 17, N. 2 (2007), 403--422.


\bibitem{Man}
R. Ma\~n\'e. \emph{  Generic properties and problems of minimizing
measures of Lagrangian systems},   Nonlinearity, N. 9 (1996), 273--310.

\bibitem{Mar} 
R. Markarian,   \emph{ Billiards with polynomial decay of
correlations},    Ergodic Theory Dynam. Systems 24 (1) (2004),
177--197

\bibitem{Mas}    D. Massart,    \emph{
Aubry sets vs Mather sets in two degrees of freedom},    (2008) To appear
preprint.


\bibitem{Mat}    J. Mather,   \emph{ Action minimizing invariant measures for positive definite Lagrangian systems},  Math. Z., N 2 (1991), 169--207.

\bibitem{Non} 
S. Nonnenmacher,
\emph{ Some open questions in wave chaos},
Nonlinearity, 21, no. 8  (2008), T113--T121



\bibitem{Ru1} 
R. Ruggiero.   \emph{  Expansive geodesic flows in manifolds with no
conjugate points},    Ergodic Theory Dynam. Systems 17 (1) (1997),
211--225

\bibitem{Ru2} 
R. Ruggiero,   \emph{ Expansive dynamics and hyperbolic geometry},   Bul. Braz.
Math. Soc. vol. 25, n. 2 (1994), 139--172


\bibitem{Sa}     O. M. Sarig.      \emph{ Subexponential decay of correlations}, Invent. Math.
150 (2002), 629--653.

\bibitem{Str}    D. Stroock,    \emph{
Partial Differential Equations for Probabilists},  (2008), Cambridge Press,


\bibitem{You} 
L. S. Young,    \emph{ Recurrence times and rates of mixing},    Israel J.
Math. 110 (1999), 153--188.

\end{thebibliography}
\end{document}